\documentclass[12pt]{amsart}
\usepackage{fullpage,url,amssymb,enumerate,colonequals}
\usepackage{mathrsfs} % for \mathscr (script letters)

\usepackage[OT2,T1]{fontenc}
\DeclareSymbolFont{cyrletters}{OT2}{wncyr}{m}{n}
\DeclareMathSymbol{\Sha}{\mathalpha}{cyrletters}{"58}

\usepackage{color}

\newcommand{\defi}[1]{\textsf{#1}} % for defined terms

\newcommand{\Aff}{\mathbb{A}}

\newcommand{\F}{\mathbb{F}}
\newcommand{\G}{\mathbb{G}}

\newcommand{\PP}{\mathbb{P}}
\newcommand{\Q}{\mathbb{Q}}
\newcommand{\R}{\mathbb{R}}
\newcommand{\Z}{\mathbb{Z}}
\newcommand{\Qbar}{{\overline{\Q}}}

\newcommand{\kbar}{{\overline{k}}}
\newcommand{\ksep}{{k_{\textup{s}}}}

\newcommand{\Xtilde}{{\widetilde{X}}}
\newcommand{\gtilde}{{\widetilde{g}}}

\newcommand{\calO}{\mathcal{O}}
\newcommand{\calP}{\mathcal{P}}

\newcommand{\calX}{\mathcal{X}}
\newcommand{\calY}{\mathcal{Y}}

\DeclareMathOperator{\Char}{char}

\DeclareMathOperator{\coarse}{coarse}

\DeclareMathOperator{\Frac}{Frac}

\DeclareMathOperator{\Gal}{Gal}

\DeclareMathOperator{\Pic}{Pic}

\DeclareMathOperator{\Proj}{Proj}

\DeclareMathOperator{\Spec}{Spec}

\newcommand{\Sch}{Sch}

\newcommand{\fppf}{{\operatorname{fppf}}}

\newcommand{\HH}{{\operatorname{H}}}

\newcommand{\intersect}{\cap} % binary intersection
\newcommand{\isom}{\simeq}

\newcommand{\isomto}{\overset{\sim}{\rightarrow}}

\newtheorem{theorem}{Theorem}
\newtheorem{lemma}[theorem]{Lemma}

\newtheorem{proposition}[theorem]{Proposition}

\theoremstyle{definition}

\theoremstyle{remark}
\newtheorem{remark}[theorem]{Remark}

\usepackage{microtype}

\usepackage[
	pagebackref,
	pdfauthor={Manjul Bhargava and Bjorn Poonen}, % add other authors
	pdftitle={The local-global principle for integral points on stacky curves},
]{hyperref}
\usepackage[alphabetic,backrefs,lite]{amsrefs} % for bibliography

\begin{document}

\title{The local-global principle for integral points on stacky curves}

\author{Manjul Bhargava}
\address{Department of Mathematics, Princeton University, Princeton, NJ 08544, USA}
\email{bhargava@math.princeton.edu}

\author{Bjorn Poonen}
\address{Department of Mathematics, Massachusetts Institute of Technology, Cambridge, MA 02139-4307, USA}
\email{poonen@math.mit.edu}
\urladdr{\url{http://math.mit.edu/~poonen/}}

\thanks{M.B.\ was supported in part by National Science Foundation grant DMS-1001828 and Simons Foundation grant \#256108.  B.P.\ was supported in part by National Science Foundation grant DMS-1601946 and Simons Foundation grants \#402472 and \#550033.}
\date{May 27, 2020}

\begin{abstract}
We construct a stacky curve 
of genus $1/2$ (i.e., Euler characteristic~$1$) over $\Z$ 
that has an $\R$-point and a $\Z_p$-point for every prime $p$
but no $\Z$-point. 
This is best possible: we also prove that any stacky curve 
of genus less than $1/2$ over 
a ring of $S$-integers of a global field 
\emph{satisfies} the local-global principle for integral points.
\end{abstract}

\maketitle

\section{Introduction}

Let $k$ be a global field, i.e., a finite extension of either $\Q$ or $\F_p(t)$.
For each nontrivial place $v$ of $k$, 
let $k_v$ be the completion of $k$ at $v$.
Let $X$ be a smooth projective geometrically integral curve of genus $g$
over $k$.
If $X$ has a $k$-point, then of course $X$ has a $k_v$-point for every~$v$.
The converse holds if $g=0$ (by the Hasse--Minkowski theorem), 
but there are well-known counterexamples of higher genus;
in fact, counterexamples exist over every global field \cite{Poonen2010-local-global}.
This motivates the question:
What is the smallest $g$ such that there exists a counterexample
of genus $g$ over some global field?
The answer is $1$.
Indeed, the first counterexample discovered was a genus 1 curve,
the smooth projective model of $2y^2=1-17x^4$ over $\Q$ \cites{Lind1940,Reichardt1942}.
In~fact, a positive proportion of genus~$1$ curves in the weighted projective space $\PP(1,1,2)$ given by
$z^2=f(x,y)$, where $f(x,y)$ is an integral binary quartic form, violate the local-global principle 
over~$\Q$~\cite{Bhargava2013hyper}.

Let us now generalize to allow $X$ to be a \emph{stacky} curve over $k$.
(See Sections \ref{S:stacks} and~\ref{S:stacky curves} for our conventions.)
Then the genus $g$ of $X$ --- defined by the formula $\chi=2-2g$, where $\chi$ is the topological Euler characteristic of $X$ --- is no longer constrained to be a natural number;
certain \emph{fractional} values are also possible.
Therefore we may now ask: What is the smallest $g$
such that there exists a stacky curve of genus $g$ over some global field $k$
violating the local-global principle? 
It turns out that if we formulate the local-global principle using
\emph{rational} points over $k$ and its completions,
then the answer is not interesting, because rational points are
almost the same as rational points on the coarse moduli space of $X$: 
see Section~\ref{S:rational points}.
Therefore we will answer our question in the context of 
a local-global principle for \emph{integral} points on a stacky curve.

Our first theorem gives a proper stacky curve of genus 1/2 over $\Z$ 
that violates the local-global principle.

\begin{theorem}
\label{T:counterexample}
Let $p,q,r$ be primes congruent to $7 \pmod 8$
such that $p$ is a square $\pmod q$ and $\pmod r$, 
and $q$ is a square $\pmod r$.
Let $f(x,y) = ax^2+bxy+cy^2$ be a positive definite integral binary
quadratic form of 
discriminant $-pqr$ such that $a$ is a nonzero square $\pmod q$ but a
nonsquare $\pmod p$ and $\pmod r$.
Let $\calY \colonequals \Proj \Z[x,y,z]/(z^2-f(x,y))$.
Define a $\mu_2$-action on $\calY$ by letting
$\lambda \in \mu_2$ act as $(x:y:z) \mapsto (x:y:\lambda z)$.
Let $\calX$ be the quotient stack $[\calY/\mu_2]$.
Then
\begin{enumerate}[\upshape (a)]
\item the genus of $\calX$ is $1/2$ $($i.e., $\chi(\calX)=1)$;
\item $\calX(\Z_\ell) \ne \emptyset$ for every rational prime $\ell$
and $\calX(\R) \ne \emptyset$;
\item $\calX(\Z)=\emptyset$, and even $\calX(\Z[1/(2pqr)])=\emptyset$. \label{smoothness}
\end{enumerate}
The same conclusions hold if instead we define $\calX$ as $[\calY/(\Z/2\Z)]$,
where $\Z/2\Z$ acts on $\calY$ through 
the nontrivial homomorphism $\Z/2\Z \to \mu_2$;
this $\calX$ is a Deligne--Mumford stack even over $\Z$.
\end{theorem}

\begin{remark}
The hypotheses in Theorem~\ref{T:counterexample} can be satisfied.
For example, let $p=7$, $q=47$, $r=31$, and $f(x,y)=3x^2+xy+850y^2$.  
\end{remark}

\begin{remark}
The reason for considering $\Z[1/(2pqr)]$ in \eqref{smoothness} is that $\calX$ is smooth
over that base.
\end{remark}

\begin{remark}
Section~8 of \cite{Darmon-Granville1995} can be interpreted as saying
that the proper stacky curve 
\[
	\left[ \left( \Spec \frac{\Z[x,y,z]}{(x^2+29y^2-3z^3)} - \{x=y=z=0\} \right) / \G_m \right]
\]
is a similar counterexample to the local-global principle,
but of genus~$2/3$.
\end{remark}

Our second theorem shows that any stacky curve of genus less than~$1/2$ over a ring of $S$-integers of a global field 
\emph{satisfies} the local-global principle.
Let $k$ be a global field, and 
let $k_v$ denote the completion of $k$ at~$v$.
Let $S$ be a finite nonempty set of places of~$k$ containing all the archimedean places.
Let $\calO$ be the ring of $S$-integers in~$k$; that is, 
$\calO \colonequals \{x \in k : v(x) \ge 0 \textup{ for all $v \notin S$}\}$.
For each $v \notin S$, let $\calO_v$ be the completion of $\calO$ at $v$.
For each $v \in S$, let $\calO_v = k_v$.

\begin{theorem}
\label{T:teardrop}
Let $\calX$ be a stacky curve over $\calO$ of genus less than $1/2$ $($i.e., $\chi(\calX)>1$$)$.
If $\calX(\calO_v) \ne \emptyset$ for all places $v$ of $k$,
then $\calX(\calO) \ne \emptyset$.
\end{theorem}

\section{Stacks}
\label{S:stacks}

By a \defi{stack}, we mean an algebraic (Artin) stack $\calX$ over a scheme $S$ \cite{StacksProject}*{\href{http://stacks.math.columbia.edu/tag/026O}{Tag~026O}}.
For any object $T \in (\Sch/S)_{\fppf}$,
we write $\calX(T)$ for 
the set of isomorphism classes of $S$-morphisms $T \to \calX$,
or equivalently 
(by the $2$-Yoneda lemma \cite{StacksProject}*{\href{http://stacks.math.columbia.edu/tag/04SS}{Tag~04SS}}),
the set of isomorphism classes of the fiber category $\calX_T$.
If $T=\Spec A$, we write $\calX(A)$ for $\calX(T)$.

\section{Stacky curves}
\label{S:stacky curves}

Let $k$ be an algebraically closed field.
Let $X$ be a \defi{stacky curve} over $k$, 
i.e., a smooth separated irreducible $1$-dimensional Deligne--Mumford stack over $k$ 
containing a nonempty open substack isomorphic to a scheme.
(This definition is slightly more general than \cite{Voight-Zureick-Brown-preprint}*{Definition~5.2.1}
in that we require only separatedness instead of properness, to allow punctures.)

By the Keel--Mori theorem \cite{Keel-Mori1997} 
in the form given in \cite{Conrad2005preprint} and \cite{Olsson2016}*{Theorem~11.1.2},
$X$ has a morphism to a coarse moduli space $X_{\coarse}$ that is a smooth integral curve over $k$.
We have $X_{\coarse} = \Xtilde_{\coarse} - Z$
for some smooth projective integral curve $\Xtilde_{\coarse}$ 
and some finite set of closed points $Z$.
Moreover, by \cite{Olsson2016}*{Theorem~11.3.1}, 
each $P \in X_{\coarse}(k)$ has an \'etale neighborhood $U$ above which $X \to X_{\coarse}$
has the form $[V/G] \to U$ for some possibly ramified finite $G$-Galois cover $V \to U$ (by a scheme),
where $G$ is the stabilizer of~$X$ above~$P$.
The stacky curve $X$ is called \defi{tame above~$P$} if $\Char k \nmid |G|$, and \defi{tame} if it is tame above every~$P$. 
Let $\calP \subset X_{\coarse}(k)$ be the (finite) set above which the stabilizer is nontrivial;
then the morphism $X \to X_{\coarse}$ is an isomorphism above $X_{\coarse}-\calP$.

Let $\tilde{g}_{\coarse}$ be the genus of $\Xtilde_{\coarse}$; 
then the Euler characteristic $\chi(X_{\coarse})$ is $(2-2\gtilde_{\coarse}) -\#Z$.
We now follow \cite{Kobin-preprint} to define $\chi(X)$ and $g(X)$.
For $P$, $U$, $V$, $G$ as above, 
let $G_i \le G$ be the ramification subgroups for $V \to U$ above~$P$,
and define 
\[
    \delta_P \colonequals \sum_{i \ge 0} \frac{|G_i|-1}{|G|}
\]
(which simplifies to only the first term $(|G|-1)/|G|$ if $X$ is tame above $P$).
Then define the \defi{Euler characteristic} by 
\[
    \chi(X) \colonequals \chi(X_{\coarse}) - \sum_{P \in \calP} \delta_P.
\]
(This is motivated by the Riemann--Hurwitz formula.
See \cites{Voight-Zureick-Brown-preprint,Kobin-preprint} for other motivation.)
Finally, define the \defi{genus} $g=g(X)$ by $\chi(X)=2-2g$.

\begin{lemma}
\label{L:genus < 1/2}
Let $X$ be a stacky curve over an algebraically closed field $k$ with $g < 1/2$.
Then $X_{\coarse} \isom \PP^1$ and $\#\calP \le 1$ and $X$ is tame.
\end{lemma}

\begin{proof}
Since $g<1/2$, we have $\chi(X)>1$.
For each $P \in \calP$, note that $\delta_P \ge (|G|-1)/|G| \ge 1/2$.  Now
\[
	\chi(X) = 2-2\gtilde_{\coarse} - \#Z - \sum_{P \in \calP} \delta_P, 
\]
which is $\le 1$ if $\gtilde_{\coarse} \ge 1$ or $\#Z \ge 1$ or $\#\calP \ge 2$.
Thus $\gtilde_{\coarse}=0$, $\#Z=0$, and $\#\calP \le 1$. 
Furthermore, if $X$ is not tame, then there exists $P \in \calP$
with $\delta_P \ge (|G|-1)/|G|+1/|G|\geq 1$, which again forces $\chi(X) \le 1$, a contradiction.
\end{proof}

Now let $k$ be any field.
Let $\kbar$ be an algebraic closure of $k$, and let $\ksep$ be the separable closure of $k$ in $\kbar$.
By a \defi{stacky curve} over $k$, we mean an algebraic stack $X$ over $k$
such that the base extension $X_{\kbar}$ is a stacky curve over $\kbar$.
Define $\chi(X) \colonequals \chi(X_{\kbar})$
and $g(X) \colonequals g(X_{\kbar})$.

\begin{lemma}\label{L:separable}
If $X$ is a tame stacky curve over $k$, 
then the set $\calP \subset X_{\coarse}(\kbar)$ for $X_{\kbar}$ 
consists of points whose residue fields are separable over $k$. 
\end{lemma}

\begin{proof}
Let $\bar{P} \in \calP$. 
Let $P$ be the closed point of $X_{\coarse}$ associated to $\bar{P}$.
By working \'etale locally on $X_{\coarse}$, 
we may assume that $X=[V/G]$ for a smooth curve $V$ over $k$
that is a $G$-Galois cover of $X_{\coarse}$ totally tamely ramified above $P$.
Analytically locally above $P$, the tame cover is given by the equation $y^n=\pi$
for some uniformizer $\pi$ at $P \in X_{\coarse}$.
After base change to $\kbar$, however, $\pi = u \bar{\pi}^i$, 
where $u$ is a unit, $\bar{\pi}$ is a uniformizer at $\bar{P}$, and $i$ is the inseparable degree of $k(P)/k$.
Thus $V_{\kbar}$ is analytically locally given by $y^n = u \bar{\pi}^i$.
Since $V_{\kbar}$ is smooth, $i=1$.
Thus $k(P)/k$ is separable.
\end{proof}

Next, let $\calO$ be a ring of $S$-integers in a global field $k$.
By a \defi{stacky curve} $\calX$ over $\calO$, 
we mean a separated finite-type algebraic stack over $\Spec \calO$
such that $\calX_k$ is a stacky curve.
(To be as general as possible,
we do not impose Deligne--Mumford, tameness, smoothness, 
or properness conditions
on the fibers above closed points of $\Spec \calO$.)
Define $\chi(\calX) \colonequals \chi(\calX_{\kbar})$ and $g(\calX) \colonequals g(\calX_{\kbar})$.

\section{Local-global principle for rational points}
\label{S:rational points}

We now explain 
why the local-global principle for \emph{rational} points
is not so interesting.

\begin{proposition}
\label{P:rational points}
Let $k$ be a global field.
Let $X$ be a stacky curve over $k$ with $g < 1$.
If $X(k_v) \ne \emptyset$ for all nontrivial places $v$ of $k$,
then $X(k) \ne \emptyset$.
\end{proposition}

\begin{proof}
We have $0 < \chi(X) \le 2 - 2\gtilde_{\coarse}$, so $\gtilde_{\coarse}=0$.
Thus $X_{\coarse}$ is a smooth geometrically integral curve of genus~$0$.
Because of the morphism $X \to X_{\coarse}$, 
we have $X_{\coarse}(k_v) \ne \emptyset$ for every $v$.
By the Hasse--Minkowski theorem, $X_{\coarse}(k) \ne \emptyset$,
so $X_{\coarse}$ is a dense open subscheme of $\PP^1_k$.
In particular, $X_{\coarse}(k)$ is Zariski dense in $X_{\coarse}$,
and all but finitely many of these $k$-points correspond to 
$k$-points on $X$.
\end{proof}

Because of Proposition~\ref{P:rational points},
our main theorems are concerned with the local-global principle 
for \emph{integral} points.

\section{Proof of Theorem~\ref{T:counterexample}: counterexample to the local-global principle}

(a) 
Since $(\calX_\Q)_{\coarse}$ is dominated by the genus~$0$ curve $\calY_\Q$,
we have $\gtilde_{\coarse}=0$.
The action of $\mu_2$ on $\calY_{\Qbar}$ fixes exactly two $\Qbar$-points,
namely those with $z=0$; thus $\calP=2$, and $\delta_P = 1/2$ for each $P \in \calP$.
Hence $\chi(\calX) = (2-2\cdot 0) - (1/2 + 1/2) = 1$.
(Alternatively, $\chi(\calX) = \chi(\calY)/2 = 2/2 = 1$.)

\medskip

(b)
Let $R$ be a principal ideal domain.
By definition of the quotient stack,
a morphism $\Spec R \to \calX$ 
is given by a $\mu_2$-torsor $T$ equipped with a $\mu_2$-equivariant
morphism $T \to \calY$.
The torsors are classified by 
$\HH^1_{\textup{fppf}}(R,\mu_2)$,
which is isomorphic to $R^\times/R^{\times 2}$,
since $\HH^1_{\textup{fppf}}(R,\G_m) = \Pic R = 0$.
Explicitly, if $t \in R^\times$,
the corresponding $\mu_2$-torsor is $T_t \colonequals \Spec R[u]/(u^2-t)$.
Define the twisted cover
\[
	\calY_t \colonequals \Proj R[x,y,z]/(tz^2-f(x,y))
\]
with its morphism $\pi_t \colon \calY_t \to \calX$.
To give a $\mu_2$-equivariant morphism $T_t \to \calY$
is the same as giving a morphism $\Spec R \to \calY_t$.
Thus we obtain
\[
	\calX(R) = \coprod_{t \in R^{\times}} \pi_t(\calY_t(R)).
\]

For any $\ell \notin \{p,q,r\}$, the rank~$3$ form $z^2-f(x,y)$ 
has good reduction at $\ell$,
so $\calY(\F_\ell)\ne \emptyset$, and Hensel's lemma yields
$\calY(\Z_\ell) \ne \emptyset$.
Since the discriminant of $f(x,y)$ is divisible only by $p$
and not $p^2$, the form is not identically $0$ modulo~$p$,
so there exist $\bar{a},\bar{b} \in \F_p$ 
with $f(\bar{a},\bar{b}) \in \F_p^\times$.
Lift $\bar{a},\bar{b}$ to $a,b \in \Z_p$,
so $f(a,b) \in \Z_p^\times$.
Then $\calY_{f(a,b)}(\Z_p) \ne \emptyset$.  
The same argument applies at $q$ and $r$.
Since $f$ is positive definite, $\calY(\R) \ne \emptyset$.
Thus $\calX(\Z_\ell) \ne \emptyset$ for all primes~$\ell$,
and $\calX(\R) \ne \emptyset$.

\medskip

(c)
We now show that $\calX(\Z[1/(2pqr)])=\emptyset$, i.e.,
that $\calY_t(\Z[1/(2pqr)]) = \emptyset$ for all $t \in \Z[1/(2pqr)]^\times$, 
or equivalently, 
that the quadratic form $f(x,y)$ does not represent 
any element of $\Z[1/(2pqr)]^\times$ times a square in $\Z[1/(2pqr)]$.

Completing the square shows that $f$ is equivalent over $\Q$
to the diagonal form $[a,apqr]$.
If we use $u=u_v$ to denote a unit nonresidue in $\Z_v$, then

\begin{itemize}
\item over $\Q_p$, the form $f$ is equivalent to $[u,up]$ and represents
  the squareclasses $u,up$;

\item over $\Q_q$, the form $f$ is equivalent to $[1,uq]$ and
  represents the squareclasses $1,uq$;

\item over $\Q_r$, the form $f$ is equivalent to $[u,ur]$ and
  represents the squareclasses $u,ur$.
\end{itemize}
Therefore, 
\begin{itemize}
\item $f$ takes square values in $\R$ and $\Q_q$, but not in $\Q_p$ and $\Q_r$.
\item $-f$ takes square values in $\Q_p$ and $\Q_r$, but not in $\R$ and $\Q_q$.
\end{itemize}
It follows that $f$ and $-f$ together represent squares
locally at all places, but do not globally represent squares. 

We now further check that $sf$, for {\it every} factor $s$ of $pqr$,
fails to globally represent a square (by quadratic
reciprocity, $r$ is not a square $\pmod p$ and $\pmod q$, and $q$ is not
a square $\pmod p$):

\begin{itemize}
\item $pf$ takes square values in $\R$ and $\Q_q$, but not in $\Q_p$ and $\Q_r$.

\item $qf$ takes square values in $\R$ and $\Q_p$, but not in $\Q_q$ and $\Q_r$.

\item $rf$ takes square values in $\R$ and $\Q_p$, but not in $\Q_q$ and $\Q_r$.

\item $pqf$ takes square values in $\R$ and $\Q_p$, but not in $\Q_q$ and $\Q_r$.

\item $prf$ takes square values in $\R$ and $\Q_p$, but not in $\Q_q$ and $\Q_r$.

\item $qrf$ takes square values in $\R$ and $\Q_q$, but not in $\Q_p$ and $\Q_r$.

\item $pqrf$ takes square values in $\R$ and $\Q_q$, but not in $\Q_p$ and $\Q_r$.
\end{itemize}
Since $2$ is a square in $\R$, $\Q_p$, $\Q_q$, and $\Q_r$, multiplying
each of the $sf$'s in the above statements by 2 would not
change the truth of any these statements.  Meanwhile, since $-1$ and
$-2$ are nonsquares in $\R$, $\Q_p$, $\Q_q$, and $\Q_r$, 
multiplying the $sf$'s in the statements above by $-1$ or $-2$ would
simply reverse all the conditions (in particular, all would fail
to represent squares in~$\R$).  

We conclude that $\calY_t(\Z[1/(2pqr)]) = \emptyset$ for all $t \in
\Z[1/(2pqr)]^\times$, i.e., $\calX(\Z[1/(2pqr)])=\emptyset$, as~claimed.

The same arguments apply to $\calX' \colonequals [\calY/(\Z/2\Z)]$;
in particular, 
\[
    \calX'(\Z[1/(2pqr)]) = \calX(\Z[1/(2pqr)]) = \emptyset,
\]
because the homomorphism $\Z/2\Z \to \mu_2$ is an isomorphism 
over $\Z[1/2]$ and hence over $\Z[1/(2pqr)]$.

\section{Stacks over local rings}

This section contains some results to be used in 
the proof of Theorem~\ref{T:teardrop}.

\begin{proposition}
\label{P:A-point lifts}
Let $A$ be a noetherian local ring.
Let $X$ be an algebraic stack of finite type over $A$.
Let $x \in X(A)$.
Then there exists a finite-type algebraic space $U$ over $A$,
a smooth surjective morphism $f \colon U \to X$, 
and an element $u \in U(A)$ such that $f(u) = x$.
\end{proposition}

\begin{proof}
By definition,
there exists a finite-type $A$-scheme $V$ 
and a smooth surjective morphism $V \to X$.
Taking the $2$-fiber product with $\Spec A \stackrel{x}\to X$
yields an algebraic space $V_x \to \Spec A$.
Then $V_x \to \Spec A$ is smooth, so it admits \'etale local sections.
Thus we can find a Galois \'etale extension $A'$ of $A$,
say with group $G$, 
such that $x$ lifts to a morphism $\Spec A' \stackrel{v}\to V$
equipped with a compatible system of isomorphisms between 
the conjugates of $v$.

Let $n=\#G$.
Let $V^n_X$ be the $2$-fiber product over $X$ of $n$ copies of $V$,
indexed by $G$.
The left translation action of $G$ on $G$
induces a right $G$-action on $V^n_X$ respecting the morphism $V^n_X \to X$,
and there is also a right $G$-action on $\Spec A'$.
Therefore we may twist $V^n_X$ to obtain 
a new algebraic space $U$ lying over $X$
(a quotient of $V^n_X \times_A A'$ by a twisted action of $G$)
such that the element of $V^n_X(A')$ given by the conjugates of $v$
and the isomorphisms between them descends to an element of $U(A)$.
\end{proof}

\begin{remark}
Atticus Christensen, combining a variant of our proof with other arguments,
has extended Proposition~\ref{P:A-point lifts} to other rings $A$,
such as arbitrary products of complete noetherian local rings,
and ad\`ele rings of global fields \cite{Christensen2020-thesis}*{Theorem~7.0.7 and Propositions~12.0.5 and~12.0.8}.
\end{remark}

For any valued field $K$, let $\widehat{K}$ denote its completion.

\begin{proposition}
\label{P:local points on algebraic space}
Let $A$ be an excellent henselian discrete valuation ring.
Let $K=\Frac A$.
Let $U$ be a separated finite-type algebraic space over $K$.
\begin{enumerate}[\upshape (a)]
\item \label{I:manifold}
The set $U(K)$ has a topology inherited from the topology on $K$.
\item \label{I:Zariski dense}
If $U$ is smooth and irreducible, 
then any nonempty open subset of $U(K)$ is Zariski dense in $U$.
\end{enumerate}
\end{proposition}

\begin{proof}
\hfill
\begin{enumerate}[\upshape (a)]
\item 
In fact, much more is true:
if $K=\widehat{K}$, then the analytification of $U$ exists 
as a rigid analytic space \cite{Conrad-Temkin2009}*{Theorem~1.2.1}.
If $K \ne \widehat{K}$, equip $U(K)$ with 
the subspace topology inherited from $U(\widehat{K})$.
\item
If $K=\widehat{K}$,
this follows from the fact that a nonzero power series in $n$ variables
over $K$ cannot vanish on a nonempty open subset of $K^n$.
If $K \ne \widehat{K}$, use Artin approximation: 
any point of $U(\widehat{K})$ can be approximated by a point of $U(K)$.\qedhere
\end{enumerate}
\end{proof}

\begin{proposition}
\label{P:U(A) is open in U(K)}
Let $A$ be an excellent henselian discrete valuation ring.
Let $K=\Frac A$.
Let $U$ be a separated finite-type algebraic space over $A$.
Then $U(A)$ is an open subset of $U(K)$.
\end{proposition}

\begin{proof}
Since $U$ is separated over $A$, the map $U(A) \to U(K)$ is injective.
Let $u \in U(A)$.
Choose a separated $A$-scheme $V$
with an \'etale surjective morphism $f \colon V \to U$.
Then $u$ lifts to some $v \in V(A')$ for some finite \'etale $A$-algebra $A'$.
Let $K'=\Frac A'$.
Since $V$ is a separated $A$-scheme,
$V(A')$ is an open subset of $V(K')$.
If $A$ is complete, then 
the \'etale morphism $V \to U$ induces an \'etale morphism of 
analytifications \cite{Conrad-Temkin2009}*{Theorem~2.3.1},
so $V(K') \to U(K')$ is a local homeomorphism;
in particular, it defines a homeomorphism 
from a neighborhood $N_V$ of $v$ in $V(K')$ 
to a neighborhood $N_U$ of $u$ in $U(K')$,
and we may assume that $N_V \subseteq V(A')$.
In the general case, a given point of $V(\widehat{K'})$ maps to some point
of $U(K')$ if and only if it is in $V(K')$,
so the homeomorphism for $\widehat{K'}$-points
restricts to a homeomorphism for $K'$-points,
which we again denote $N_V \isomto N_U$.
If $u_1 \in N_U \intersect U(K)$,
then $u_1$ lies in the image of $N_V \subseteq V(A')$,
so $u_1 \in U(A')$;
now $u_1 \in U(A') \intersect U(K)$, which is $U(A)$
since $U$ is a sheaf on $(\Spec A)_{\fppf}$.
Hence $U(A)$ is open in $U(K)$.
\end{proof}

\section{Proof of Theorem~\ref{T:teardrop}}

By Lemma~\ref{L:genus < 1/2}, we have
$(\calX_\kbar)_{\coarse}\isom \PP^1_\kbar$, and hence $(\calX_k)_{\coarse}$ is a smooth proper curve of genus~$0$. 
Since $\calX$ has an $\calO_v$-point for every $v$,
the stack $\calX_k$ has a $k_v$-point for every $v$,
so $(\calX_k)_{\coarse}$ has a $k_v$-point for every $v$. Thus $(\calX_k)_{\coarse} \isom \PP^1_k$.

If $\calX_k \to (\calX_k)_{\coarse}$ is not an isomorphism,
then by Lemma~\ref{L:genus < 1/2}, there is a unique $\kbar$-point 
above which it fails to be an isomorphism,
and by Lemma~\ref{L:separable}, it is a $\ksep$-point, 
and that point must be $\Gal(\ksep/k)$-stable, hence a $k$-point of $\PP^1$,
which we may assume is $\infty$.
Thus $\calX_k$ contains an open substack isomorphic to $\Aff^1_k$.

Since all the stacks are of finite presentation,
the isomorphism just constructed extends 
above some affine open neighborhood of the generic point in $\Spec \calO$.
That is, there exists a finite set of places $S' \supseteq S$ 
such that if $\calO'$ is the ring of $S'$-integers in $k$,
then the stack $\calX_{\calO'}$ contains 
an open substack isomorphic to $\Aff^1_{\calO'}$.

Let $v \in S'-S$.
Let $\calO_{(v)}$ be the localization of $\calO$ at $v$,
and let $\calO_{v,h}$ be its henselization in $\calO_v$,
so $\calO_{v,h}$ is the set of elements of $\calO_v$ 
that are algebraic over $k$.
Let $k_{v,h} = \Frac \calO_{v,h}$.
We are given $x \in \calX(\calO_v)$.
Let $U$, $f$, and $u$ be as in Proposition~\ref{P:A-point lifts}
with $A=\calO_v$.
By Proposition~\ref{P:U(A) is open in U(K)},
$U(\calO_v)$ is open in $U(k_v)$.
Let $U_0$ be the connected component of $U_{k_v}$ containing $u$,
so $U_0(k_v)$ is open in $U(k_v)$.
The morphisms $U_0 \to U_{k_v} \to \calX_{k_v} \to \Spec k_v$ are smooth, 
so $U_0$ is smooth and irreducible. Therefore, by 
Proposition~\ref{P:local points on algebraic space}\eqref{I:Zariski dense},
the set $U(\calO_v) \intersect U_0(k_v)$ is Zariski dense in~$U_0$.
On the other hand, $U_0$ dominates $\calX_{k_v}$ since 
$U_0 \to \calX_{k_v}$ is smooth and $\calX_{k_v}$ is irreducible.
By the previous two sentences,
there exists $u_0 \in U(\calO_v) \intersect U_0(k_v)$ 
mapping into the subset $\Aff^1(k_v)$ of $\calX(k_v)$.
By Artin approximation, we may replace $u_0$ by a nearby point
to assume also that $u_0 \in U(\calO_{v,h})$.

Let $U_1$ be the inverse image of $\Aff^1_{k_{v,h}}$ 
under $U_{k_{v,h}} \to \calX_{k_{v,h}}$.
By Proposition~\ref{P:U(A) is open in U(K)}, 
$U(\calO_{v,h})$ is open in $U(k_{v,h})$,
so $U(\calO_{v,h}) \intersect U_1(k_{v,h})$ 
is an open neighborhood of $u_0$ in $U_1(k_{v,h})$.
Since $U_1 \to \Aff^1_{k_{v,h}}$ is smooth, 
the image of this neighborhood is a
nonempty open subset $B_v$ of $\Aff^1(k_{v,h})$.
By construction, $B_v$ is contained in the image of
$U(\calO_{v,h}) \to \calX(\calO_{v,h}) \subseteq \calX(k_{v,h})$,
so $B_v \subseteq \calX(\calO_{v,h})$.

By strong approximation, there exists $x \in \Aff^1(\calO')$
such that $x \in B_v$ for all $v \in S'-S$.
For each $v \in S'-S$, since $B_v \subseteq \calX(\calO_{v,h})$, 
there exists $x_v \in \calX(\calO_{v,h})$
such that $x$ and $x_v$ become equal in $\calX(k_{v,h})$.
Finally, the following lemma shows that $x$ comes from 
an element of $\calX(\calO)$.

\begin{lemma}
\label{L:last}
If $x \in \calX(\calO')$
and $x_v \in \calX(\calO_{v,h})$ for each $v \in S'-S$
are such that the images of $x$ and $x_v$ in $\calX(k_{v,h})$ are equal
for every $v \in S'-S$,
then there exists an element of $\calX(\calO)$
mapping to $x$ in $\calX(\calO')$
and to $x_v$ in $\calX(\calO_{v,h})$ for each $v \in S'-S$.
\end{lemma}

\begin{proof}
Since $\calX$ is of finite presentation over $\calO$,
the element $x_v$ comes from an element $\widetilde{x}_v$
of some finitely generated $\calO$-subalgebra $A_v$ of $\calO_{v,h}$.
The schemes $\Spec A_v$ together with $\Spec \calO'$
form an fppf covering of $\Spec \calO$,
so the stack property of $\calX$ shows that $x$ and the $\widetilde{x}_v$
come from an element of $\calX(\calO)$.
\end{proof}

\begin{remark}
Inspired by an earlier draft of our article, 
Christensen has found a natural way to define a topology on 
the set of adelic points of a finite-type algebraic stack,
and has proved a strong approximation theorem for a stacky curve
with $\chi > 1$ \cite{Christensen2020-thesis}*{Theorem~13.0.6}.
His argument can substitute for the three paragraphs 
before Lemma~\ref{L:last} and hence give a partially independent
proof of Theorem~\ref{T:teardrop}.
\end{remark}

%****************************************************************************
\section*{Acknowledgements} 
 
We thank Johan de Jong, Martin Olsson, Ashvin Swaminathan, Martin Ulirsch, John Voight,
and David Zureick-Brown for discussions.

\end{document}